\documentclass[11pt,a4paper,reqno]{amsart}

\usepackage[utf8]{inputenc}
\usepackage{amsmath}
\usepackage{amssymb}
\usepackage{amsfonts}
\usepackage[toc,page]{appendix}
\usepackage{hyperref}
\usepackage{tikz}
\usetikzlibrary{decorations.pathreplacing}
\usepackage{color}
\usepackage{graphicx}
\usepackage{mathtools}
\usepackage[none]{hyphenat}
\usepackage{ytableau}
\usepackage{rotating}

\newcommand{\qbinom}[2]{\genfrac{[}{]}{0pt}{}{#1}{#2}_{q}}

\newtheorem{theorem}{Theorem}
\newtheorem*{theorem*}{Theorem}
\theoremstyle{plain}

\newtheorem{corollary}{Corollary}

\newtheorem*{lemma*}{Lemma}

\newtheorem*{proposition*}{Proposition}
\theoremstyle{remark}

\newtheorem*{remark*}{Remark}
\theoremstyle{definition}
\newtheorem*{acknowledgements}{Acknowledgements}

\begin{document}


\title{Toeplitz minors and specializations of skew Schur polynomials}
\author{David Garc\'{\i}a-Garc\'{\i}a}
\author{Miguel Tierz}
\address[$\dagger$]{Departamento de Matem\'{a}tica, Grupo de F\'{\i}sica Matem\'{a}%
tica, Faculdade de Ci\^{e}ncias, Universidade de Lisboa, Campo Grande, Edif%
\'{\i}cio C6, 1749-016 Lisboa, Portugal.}
\email{dgarciagarcia@fc.ul.pt}
\address[$\ddag$]{Departamento de Matem\'{a}tica, ISCTE - Instituto Universit\'{a}rio de Lisboa, Avenida
das Forças Armadas, 1649-026 Lisboa, Portugal.}
\email{mtpaz@iscte-iul.pt}
\address[$\ddag$]{Departamento de Matem\'{a}tica, Grupo de F\'{\i}sica Matem\'{a}%
tica, Faculdade de Ci\^{e}ncias, Universidade de Lisboa, Campo Grande, Edif%
\'{\i}cio C6, 1749-016 Lisboa, Portugal.}
\email{tierz@fc.ul.pt}

\begin{abstract}
We express minors of Toeplitz matrices of finite and large dimension in terms of symmetric functions. Comparing the resulting expressions with the inverses of some Toeplitz matrices, we obtain explicit formulas for a Selberg-Morris integral and for specializations of certain skew Schur polynomials.

\end{abstract}

\maketitle

\section{Introduction}

Let $f(e^{i\theta})=\sum_{k\in\mathbb{Z}}d_{k}e^{ik\theta}$ be an integrable function on the unit circle. The Toeplitz matrix generated by $f$ is the matrix
\begin{equation*}
	T(f)=(d_{j-k})_{j,k\geq1}.
\end{equation*}
That is, $T(f)$ is an infinite matrix, constant along its diagonals, which entries are the Fourier coefficients of the function $f$. We denote by $T_{N}(f)$ its principal submatrix of order $N$, and
\begin{equation*}
    D_{N}(f)=\det{T_{N}(f)}.
\end{equation*}
This determinant has the following integral representation
\begin{equation*}
	D_{N}(f)=\int_{U(N)}f(M)dM=\frac{1}{N!}\frac{1}{(2\pi)^N}\int_{0}^{2\pi}...\int_{0}^{2\pi}\prod_{j=1}^{N}f(e^{i\theta_{j}})\prod_{1\leq j<k\leq N}|e^{i\theta_{j}}-e^{i\theta_{k}}|^{2}d\theta_{1}... d\theta_{N},
\end{equation*}
where $dM$ denotes the normalized Haar measure on the unitary group $U(N)$. This is known as Heine identity. A main result in the theory of Toeplitz matrices is the strong Szeg\H o limit theorem, that describes the behaviour of these determinants as $N$ grows to infinity, as long as the function $f$ is sufficiently regular (see section \ref{s.toep} for a precise statement of the theorem).

~

A Toeplitz minor is a minor of a Toeplitz matrix, obtained by striking a finite number of rows and columns from $T(f)$. This can be realized, up to a sign, as the determinant of a matrix of the form
\begin{equation} \label{tmin}
	T_{N}^{\lambda,\mu}(f)=(d_{j-\lambda_{j}-k+\mu_{k}})_{j,k=1}^{N},
\end{equation}
where $\lambda$ and $\mu$ are integer partitions that encode the particular striking considered (see section \ref{s.symm} for more details). We denote
\begin{equation*}
    D_{N}^{\lambda,\mu}(f)=\det{T_{N}^{\lambda,\mu}(f)}.
\end{equation*}
Toeplitz minors also have an integral representation \cite{BumpDiaconis,AdlervanMo}
\begin{align}
	D_{N}^{\lambda,\mu}(f) &= \int_{U(N)}  \overline{s_{\lambda}(M)}s_{\mu}(M)f(M) dM = \label{heinemin} \\ \frac{1}{N!}\frac{1}{(2\pi)^N}&\int_{0}^{2\pi}...\int_{0}^{2\pi}s_{\lambda}(e^{-i\theta_{1}},...,e^{-i\theta_{N}})s_{\mu}(e^{i\theta_{1}},...,e^{i\theta_{N}})\prod_{j=1}^{N}f(e^{i\theta_{j}})\prod_{1\leq j<k\leq N}|e^{i\theta_{j}}-e^{i\theta_{k}}|^{2}d\theta_{1}... d\theta_{N}, \nonumber
\end{align}
where $s_{\lambda},s_{\mu}$ are Schur polynomials\footnote{We abuse notation here; we assume it is clear when the expression $f(M)$ should be read as $\prod_{j}f(e^{i\theta_{j}})$ (i.e. when $f$ is a function on the unit circle) and when it should be read as $f(e^{i\theta_{1}},\dots,e^{i\theta_{N}})$ (i.e. when $f$ is a symmetric function in several variables). See \cite{MacDonald} and section \ref{s.symm} for definitions of Schur polynomials.}. Bump and Diaconis \cite{BumpDiaconis} described the asymptotic behaviour of Toeplitz minors generated by functions that are sufficiently regular, as in Szeg\H o's theorem. They proved that in the large $N$ limit, these minors can be expressed as the product of the corresponding Toeplitz determinant times a ``combinatorial" factor, that depends only on the function $f$ and the striking considered and is independent of $N$ (see section \ref{s.toep} for a precise statement). Tracy and Widom obtained a similar result in \cite{TWmin}, and they were compared in \cite{Deha2}. Further generalizations regarding the asymptotics of integrals of the type \eqref{heinemin} were given in  \cite{Deha1,Lyons}.  Other works that study Toeplitz minors in relationship with Schur and skew Schur polynomials are \cite{Lascoux,Alexandersson,Okounkov,Maximenko}

~

The asymptotics of Toeplitz determinants generated by symbols that do not verify the regularity conditions in Szeg\H o's theorem have been long studied. In the seminal work \cite{FisherHartwig}, Fisher and Hartwig conjectured the asymptotic behaviour of Toeplitz determinants generated by a class of (integrable) functions that violate these conditions. The functions in this class are products of a function which is regular, in the sense of Szeg\H o's theorem, and a finite number of so-called pure Fisher-Hartwig singularities. Their conjecture was later refined in \cite{Basor} and \cite{BasorTracy}, and only recently a complete description of the asymptotics of these determinants was achieved by Deift, Its and Krasovsky \cite{DIKFH}. See \cite{DIK} for a detailed historical account of the subject.

~

In this paper we exploit the formalism of symmetric functions to study Toeplitz minors. After section \ref{s.prel}, where some known results are reviewed, we obtain an equivalent expression for the combinatorial factor of Bump and Diaconis in terms of skew Schur polynomials. This is done in section \ref{s.res}, where we also characterize $(i)$ a class of Toeplitz minors for which an exact asymptotic expression can be obtained, and $(ii)$ a class of Toeplitz minors that can be realized as the specialization of a single skew Schur polynomial. In section \ref{s.inv} we compute the inverses of some Toeplitz matrices, using the Duduchava-Roch formula and the kernel associated to two sets of biorthogonal polynomials on the unit circle. Comparing these matrices with their expressions in terms of Toeplitz minors we obtain explicit evaluations of a family of specialized skew Schur polynomials and of a Selberg-Morris type integral.


\section{Preliminaries} \label{s.prel}

\subsection{Symmetric functions} \label{s.symm}

Let us recall some basic results involving symmetric functions that can be found in \cite{MacDonald,Stanley}, for example. We denote $z=e^{i\theta}$ in the following, and treat $z$ as a formal variable. A partition $\lambda=(\lambda_{1},\dots,\lambda_{l})$ is a finite and non-increasing sequence of positive integers. The number of nonzero entries is called the length of the partition and is denoted by $l(\lambda)$, and the sum $|\lambda|=\lambda_{1}+\dots+\lambda_{l(\lambda)}$ is called the weight of the partition. The entry $\lambda_{j}$ is understood to be zero whenever the index $j$ is greater than the length of the partition. The notation $(a^{b})$ stands for the partition with exactly $b$ nonzero entries, all equal to $a$. A partition can be represented as a Young diagram, by placing $\lambda_{j}$ left-justified boxes in the $j$-th row of the diagram. The conjugate partition $\lambda'$ is then obtained as the partition which diagram has as rows the columns of the diagram of $\lambda$ (see figure \ref{f.part} for an example).
\ytableausetup{boxsize=0.35cm}
\begin{figure}
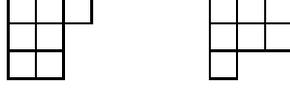
 \label{f.part}
    $\ydiagram{3,2,2} \qquad  \qquad \ydiagram{3,3,1}$
\caption{The partition $(3,2,2)$ and its conjugate $(3,3,1)$.}
\end{figure}
The following procedure describes how to obtain the Toeplitz minor $D_{N}^{\lambda,\mu}(f)$, given by \eqref{tmin}, from the underlying Toeplitz matrix (we assume in the following that the length of the partitions $\lambda$ and $\mu$ is less than or equal to $N$, the size of the matrix under consideration):
\begin{itemize}
    \item Strike the first $|\lambda_{1}-\mu_{1}|$ columns or rows of $T_{N+\max{\{\lambda_{1},\mu_{1}\}}}(f)$, depending on whether $\lambda_{1}-\mu_{1}$ is greater or smaller than zero, respectively.
	\item Keep the first row of the matrix, and strike the next $\lambda_{1}-\lambda_{2}$ rows. Keep the next row, and strike the next $\lambda_{2}-\lambda_{3}$ rows. Continue until striking $\lambda_{l(\lambda)}-\lambda_{l(\lambda)+1}=\lambda_{l(\lambda)}$ rows.
	\item Repeat the previous step on the columns of the matrix with $\mu$ in place of $\lambda$. The resulting matrix is precisely $T_{N}^{\lambda,\mu}(f)$, as defined in \eqref{tmin}.
\end{itemize}

If $x=(x_{1},x_{2},...)$ is a set of variables, the power-sum symmetric polynomials $p_{k}$ are defined as $p_{k}(x)=x_{1}^{k}+x_{2}^{k}+\dots$ for every $k\geq1$, and $p_{0}(x)=1$. They are related to the elementary symmetric polynomials $e_{k}(x)$ and the complete homogeneous polynomials $h_{k}(x)$ by the formulas
\begin{align} \label{generat} \begin{split}
    \exp{\left(\sum_{k=1}^{\infty}\frac{p_{k}(x)}{k}z^{k}\right)} = \sum_{k=0}^{\infty} h_{k}(x)z^{k} = \prod_{j=1}^{\infty}\frac{1}{1-x_{j}z} = H(x;z), \\
    \exp{\left(\sum_{k=1}^{\infty}(-1)^{k+1}\frac{p_{k}(x)}{k}z^{k}\right)} = \sum_{k=0}^{\infty} e_{k}(x)z^{k} = \prod_{j=1}^{\infty}(1+x_{j}z) = E(x;z). \\
\end{split} \end{align}
We also set $p_{k}(x)=h_{k}(x)=e_{k}(x)=0$ for negative $k$. The families $(h_{k}(x))$ and $(e_{k}(x))$ where $k\geq0$ consist of algebraically independent functions. Hence, we will see $H$ and $E$ as arbitrary functions on the unit circle depending on the parameters $x$, and we will use indistinctly their infinite product expression. We note also that these two functions satisfy $H(x;z)E(x;-z)=1$. The classical Jacobi-Trudi identities express Schur polynomials as Toeplitz minors generated by $H$ and $E$
\begin{align*}
	s_{\mu}(x) &= \det\left(h_{j-k+\mu_{k}}(x)\right)_{j,k=1}^{N}=D_{N}^{\varnothing,\mu}\left(H(x;z)\right), \\
	s_{\mu'}(x) &= \det\left(e_{j-k+\mu_{k}}(x)\right)_{j,k=1}^{N}=D_{N}^{\varnothing,\mu}\left(E(x;z)\right),
\end{align*}
where $l(\mu),l(\mu')\leq N$, respectively, and $\varnothing$ denotes the empty partition. More generally, skew Schur polynomials can be expressed as the minors
\begin{equation} \label{skewschur}
    s_{\mu/\lambda}(x) = D_{N}^{\lambda,\mu}(H(x;z)),\qquad s_{(\mu/\lambda)'}(x) = D_{N}^{\lambda,\mu}(E(x;z)),
\end{equation}
where $l(\mu),l(\mu')\leq N$ respectively. A skew Schur polynomial vanishes if $\lambda\nsubseteq\mu$, which can be seen as a consequence of its Toeplitz minor representation and the fact that the Toeplitz matrices above are triangular. A central result in the theory of symmetric functions is the Cauchy identity, and its dual form
\begin{equation*}
    \sum_{\nu}s_{\nu}(x)s_{\nu}(y)=\prod_{j=1}^{\infty}\prod_{k=1}^{\infty}\frac{1}{1-x_{j}y_{k}}, \qquad \sum_{\nu}s_{\nu}(x)s_{\nu'}(y)=\prod_{j=1}^{\infty}\prod_{k=1}^{\infty}(1+x_{j}y_{k}),
\end{equation*}
where $y=(y_{1},y_{2},\dots)$ is another set of variables and the sums run over all partitions $\nu$.

~

Gessel \cite{Gessel} obtained the following expression for the Toeplitz determinant generated by the function $f(z)=H(y;z^{-1})H(x;z)$
\begin{equation} \label{gessel}
    D_{N}\left( \prod_{k=1}^{\infty}\frac{1}{1-y_{k}z^{-1}}\prod_{j=1}^{\infty}\frac{1}{1-x_{j}z} \right) = \sum_{l(\nu)\leq N}s_{\nu}(y)s_{\nu}(x),
\end{equation}
where the sum runs over all partitions $\nu$ of length $l(\nu)\leq N$. If one of the sets of variables $x$ or $y$ is finite, say $y=(y_{1},\dots,y_{d})$, comparing the right hand side above with the sum in Cauchy identity and recalling that the Schur polynomial $s_{\nu}(y_{1},...,y_{d})$ vanishes if $l(\nu)>d$ one obtains a well known identity of Baxter \cite{Baxter}
\begin{equation} \label{baxterid}
    D_{N}\left(\prod_{k=1}^{d}\frac{1}{1-y_{k}z^{-1}}\prod_{j=1}^{\infty}\frac{1}{1-x_{j}z}\right) = \prod_{k=1}^{d}\prod_{j=1}^{\infty}\frac{1}{1-x_{j}y_{k}},
\end{equation}
valid when $N\geq d$. Note that the right hand side above is independent of $N$. An analogous identity follows if the factor $H(x;z)$ is replaced by $E(x;z)$, using the dual Cauchy identity instead. However, no such identity is available for Toeplitz determinants generated by symbols of the type $E(y;z^{-1})E(x;z)$; this will be relevant later. 


\subsection{Toeplitz determinants and minors generated by smooth symbols.} \label{s.toep}
We record now precise statements of the strong Szeg\H o limit theorem and of its generalization to Toeplitz minors.
\begin{theorem*}[Szeg\H o]
Let $f(e^{i\theta})=\sum_{k\in\mathbb{Z}}d_{k}e^{ik\theta}$ be a function on the unit circle, and suppose it can be expressed as $f(e^{i\theta})=\exp(\sum_{k\in\mathbb{Z}}c_{k}e^{ik\theta})$, where the coefficients $c_{k}$ verify
\begin{equation}
    \sum_{k\in\mathbb{Z}}|c_{k}|<\infty,\qquad \sum_{k\in\mathbb{Z}}|k||c_{k}|^{2}<\infty. \label{szcond}
\end{equation}
Then, as $N\rightarrow\infty$,
\begin{equation*}
	D_{N}(f)\sim \exp\left(Nc_{0}+\sum_{k=1}^{\infty}kc_{k}c_{-k}\right).
\end{equation*}
\end{theorem*}
A function $f$ satisfying the hypotheses of this theorem is continuous, nonzero, and has winding number zero \cite{BottSilb}. Functions with Fisher-Hartwig singularities need not verify these properties (see section \ref{s.FH}). Under these same conditions, the following theorem holds.
\begin{theorem*}[Bump, Diaconis \cite{BumpDiaconis}] \label{ssltbd}
Let $f$ verify the hypotheses in the previous theorem,
and suppose $\lambda$ and $\mu$ are partitions of weights $n$ and $m$ respectively. Then, as $N\rightarrow\infty$
\begin{equation} \label{BD}
	D_{N}^{\lambda,\mu}(f) \sim D_{N}(f) \sum_{\phi\vdash n}\sum_{\psi\vdash m}  \chi^{\lambda}_{\phi}\:\chi^{\mu}_{\psi}\:z_{\phi}^{-1}z_{\psi}^{-1} \Delta(f,\phi,\psi),
\end{equation}
where the sum runs over all the partitions $\phi$ of $n$ and $\psi$ of $m$, the terms $z_{\phi},z_{\psi}$ are the orders of the centralizers of the equivalence classes of the symmetric groups $S_{n},S_{m}$ indexed by $\phi$ and $\psi$ respectively, the functions $\chi^{\lambda},\chi^{\mu}$ are the characters associated to the irreducible representations of $S_{n}$ and $S_{m}$ indexed by $\lambda$ and $\mu$ respectively, and
\begin{equation*}
	\Delta(f,\phi,\psi) = \prod_{k=1}^{\infty} \begin{dcases}
    k^{n_{k}}c_{-k}^{n_{k}-m_{k}}m_{k}!L_{m_{k}}^{(n_{k}-m_{k})}(-kc_{k}c_{-k}), & \textrm{if }n_{k}\geq m_{k} \\
    k^{m_{k}}c_{k}^{m_{k}-n_{k}}n_{k}!L_{n_{k}}^{(m_{k}-n_{k})}(-kc_{k}c_{-k}), & \textrm{if }n_{k}\leq m_{k}
  \end{dcases}.
\end{equation*}
Above, the coefficients $n_{k},m_{k}$ correspond to the partitions $\phi=(1^{n_{1}}2^{n_{2}}\dots)$ and $\psi=(1^{m_{1}}2^{m_{2}}\dots)$ in their frequency notation, and $L_{n}^{(a)}$ are the Laguerre polynomials \cite{Szego}.
\end{theorem*}
Note that the product in the factor $\Delta(f,\phi,\psi)$ is actually finite, since only a finite number of $n_{k}$'s and $m_{k}$'s are distinct from zero for each pair $\phi,\psi$. As mentioned before, we see that in the $N\rightarrow\infty$ limit the Toeplitz minor generated by a regular symbol factors as the corresponding Toeplitz determinant times a sum depending only on $f$ and the partitions $\lambda,\mu$ (and not on $N$). The formula \eqref{BD} can be implemented in MatLab for example, leading to quick evaluations for values of, say, $|\lambda|,|\mu|= 15$. Table \ref{tab.gen} shows some of these values for particular choices of $\lambda$ and $\mu$.

\ytableausetup{boxsize=0.2cm}
\begin{table}
\begin{center}
\begin{tabular}{|c|c|l|c|c|l|}
\hline
$\lambda$ & $\mu$ & $\lim_{N\rightarrow\infty} D_{N}^{\lambda,\mu}(f)/D_{N}(f)$ & $\lambda$ & $\mu$ & $\lim_{N\rightarrow\infty} D_{N}^{\lambda,\mu}(f)/D_{N}(f)$\\
\hline

$\varnothing$ & $\ydiagram{1}$ & $c_{1}$ & $\varnothing$ & $\ydiagram{2}$ & $\frac{1}{2}c_{1}^{2}+c_{2}$\\
$\varnothing$ & $\ydiagram{1,1}$ & $\frac{1}{2}c_{1}^{2}-c_{2}$ & $\varnothing$ & $\ydiagram{3}$ & $\frac{1}{6}c_{1}^{3}+c_{1}c_{2}+c_{3}$ \\
$\varnothing$ & $\ydiagram{1,1,1}$ & $\frac{1}{6}c_{1}^{3}-c_{1}c_{2}+c_{3}$ & $\varnothing$ & $\ydiagram{2,2}$ & $\frac{1}{12}c_{1}^{4}-c_{1}c_{3}+c_{2}^{2}$ \\
\hline 
\multicolumn{6}{c}{} \\
\hline 
$\lambda$ & $\mu$ & \multicolumn{4}{l|}{$\lim_{N\rightarrow\infty} D_{N}^{\lambda,\mu}(f)/D_{N}(f)$} \\
\hline
$\ydiagram{1,1}$ & $\ydiagram{1,1}$ & \multicolumn{4}{l|}{$\frac{1}{4}c_{-1}^{2}c_{1}^{2}+c_{-1}c_{1}-\frac{1}{2}c_{-2}c_{1}^{2}-\frac{1}{2}c_{-1}^{2}c_{2}+c_{-2}c_{2}+1$} \\
$\ydiagram{1}$ & $\ydiagram{3}$ &  \multicolumn{4}{l|}{$\frac{1}{6}c_{-1}c_{1}^{3}+\frac{1}{2}c_{1}^{2}+c_{-1}c_{1}c_{2}+c_{2}+c_{-1}c_{3}$} \\
\hline

\end{tabular}
\vspace{0.5 cm}
\caption{Some values of the formula \eqref{BD}.}
\label{tab.gen}
\end{center}
\end{table}


\section{Toeplitz minors generated by symmetric functions} \label{s.res}

Let us now obtain an equivalent expression for the asymptotic formula \eqref{BD} for the case of Toeplitz minors generated by formal power series.
\begin{theorem} \label{mainth}
Let
\begin{equation*}
    f(z)=H(x;z)H(y;z^{-1}),
\end{equation*}
for some sets of variables $x$ and $y$, where $H$ is given by \eqref{generat}, and assume moreover that the sequences of complete homogeneous symmetric polynomials $(h_{k}(x))$ and $(h_{k}(y))$ are square summable. Then, for any two fixed partitions $\lambda$ and $\mu$ we have
\begin{equation} \label{sumth}
    \lim_{N\to\infty}D_{N}^{\lambda,\mu}(f) = \sum_{\nu}s_{\lambda/\nu}(y)s_{\mu/\nu}(x)\lim_{N\rightarrow\infty}D_{N}(f).
\end{equation}
Note that we understand $f$ as a formal Laurent power series whose coefficients are symmetric functions on $x$ and $y$, and thus the convergence above is in the algebra of formal power series.
\end{theorem}

\begin{proof}
Let us first observe that the limit $\lim_{N\rightarrow\infty}D_{N}(f)$ in the right hand side of \eqref{sumth} is well defined as a formal expression, since by the identities of Gessel and Cauchy we have
\begin{equation}
    \lim_{N\rightarrow\infty}D_{N}(f) = \lim_{N\rightarrow\infty}\sum_{l(\nu)\leq N}s_{\nu}(x)s_{\nu}(y) = \prod_{j=1}^{\infty}\prod_{k=1}^{\infty}\frac{1}{1-x_{j}y_{k}}.
\end{equation}
We will use the following lemma, that has an elementary proof.
\begin{lemma*}
Let $\nu$ be a partition verifying $\nu\subset(d^{N})$, and consider the partition $\overleftarrow{\nu}^{d}=(d-\nu_{N},\dots,d-\nu_{1})$ which is obtained by rotating 180º the complement of $\nu$ in the diagram of the rectangular partition $(d^{N})$. Then the Schur polynomial $s_{\nu}$ verifies
\begin{equation*}
    s_{\nu}(x_{1}^{-1},\dots,x_{N}^{-1})=s_{\overleftarrow{\nu}^{d}}(x_{1},\dots,x_{N})\prod_{j=1}^{N}x_{j}^{-d}.
\end{equation*}
\end{lemma*}

If $R,S$ are two strictly increasing sequences of natural numbers, we denote by $\det_{R,S}M$ the minor of the matrix $M$ obtained by taking the rows and columns of $M$ indexed by $R$ and $S$, respectively. Using the above lemma with $d=\max{\{\lambda_{1},\mu_{1}\}}$ we see that
\begin{align*}
    D_{N}^{\lambda,\mu}(f) = \int_{U(N)}\overline{s_{\lambda}(M)}s_{\mu}(M)f(M)dM = \int_{U(N)}s_{\overleftarrow{\lambda}^{d}}(M)\overline{s_{\overleftarrow{\mu}^{d}}(M)}f(M)dM = \det_{R,S} T(f),
\end{align*}
where the sequences $R,S$ are given by $R=(r_{j})_{j=1}^{N}=(j+\mu_{N+1-j})_{j=1}^{N}$ and $S=(s_{k})_{k=1}^{N}=(k+\lambda_{N+1-k})_{k=1}^{N}$. Since the Toeplitz matrices generated by each of the factors of $f$ verify $T(f(z))=T(H(y;z^{-1}))T(H(x;z))$, the use of Cauchy-Binet formula gives
\begin{equation} \label{CBth}
    \det_{R,S} T(f(z)) = \sum_{T} \det_{R,T}T(H(y;z^{-1}))\det_{T,S}T(H(x;z)),
\end{equation}
where the summation is over all the strictly increasing sequences $T=(t_{1},\dots,t_{N})$ of length $N$ of positive integers\footnote{We are actually using the infinite dimensional generalization of the Cauchy-Binet formula that appears in \cite{TWwords}. The convergence of the sum in the right hand side follows from the square integrability of $(h_{k}(x))$ and $(h_{k}(y))$.}. There is a correspondence between such sequences and partitions $\nu$ of length $l(\nu)\leq N$, given by $\nu_{N+1-j}=t_{j}-j$, for $j=1,...,N$. Thus, for each $T$ we have
\begin{equation*}
    \det_{T,S} T(H(x;z))  =\det(h_{t_{j}-s_{k}}(x))_{j,k=1}^{N}=\det(h_{j+\nu_{N+1-j}-k-\lambda_{N+1-k})}(x))_{j,k=1}^{N}.
\end{equation*}
Reversing the order of its rows and columns, we see that the last determinant above is $D_{N}^{\lambda,\nu}(H(x;z))$. According to \eqref{skewschur} this is precisely the skew Schur polynomial $s_{\nu/\lambda}(x)$, and an analogous derivation yields $\det_{R,T}T(H(y;z^{-1}))=s_{\nu/\mu}(y)$. We thus obtain
\begin{equation} \label{minorsum}
    D_{N}^{\lambda,\mu}(f) = \sum_{l(\nu)\leq N}s_{\nu/\mu}(y)s_{\nu/\lambda}(x).
\end{equation}
Combining this with the following identity between Schur and skew Schur polynomials (Ex. I.5.26 in \cite{MacDonald})
\begin{equation}
    \sum_{\nu}s_{\nu/\mu}(y)s_{\nu/\lambda}(x) = \sum_{\nu}s_{\lambda/\nu}(y)s_{\mu/\nu}(x)\sum_{\kappa}s_{\kappa}(y)s_{\kappa}(x), \label{auxmac}
\end{equation}
where the sums run over all partitions, we arrived at the desired conclusion, upon identification of the second sum in the right hand side above with the large-$N$ limit of the Toeplitz determinant generated by $f$.
\end{proof}

An analogous reasoning shows that identity \eqref{sumth} holds also for functions of the form
\begin{equation*}
    f(z)=E(x;z)E(y;z^{-1}),
\end{equation*}
after taking the conjugate of all the partitions indexing the skew Schur polynomials in the right hand side of \eqref{sumth}. 

Let us emphasize that the theorem is to be understood as an identity among symmetric functions. However, as usual in this context, one can specialize any algebraically independent family of symmetric functions to any given sequence of, say, real or complex numbers, and extend \eqref{sumth} to an identity involving more general Toeplitz matrices, as long as the formal manipulations are justified after this specialization (see \cite{Stanley,TWwords,BaikRains} for examples of this). Let us consider, for instance, a function $f$ that satisfies the regularity conditions in Szeg\H o's theorem. That is, assume $f(e^{i\theta}) = \exp{(\sum_{k}c_{k}e^{ik\theta})}$, where the coefficients $c_{k}$ satisfy the decay conditions \eqref{szcond}. Then, assuming that $c_{0}=0$ without loss of generality, we can write $f(e^{i\theta}) = f^{+}(e^{i\theta})f^{-}(e^{i\theta})$, where
\begin{equation} \label{factors}
    f^{+}(e^{i\theta}) = \exp{\left(\sum_{k>0}c_{k}e^{ik\theta}\right)} = 1+\sum_{k\geq1}d_{k}^{+}e^{ik\theta}, \qquad f^{-}(e^{i\theta}) = \exp{\left(\sum_{k<0}c_{k}e^{ik\theta}\right)} = 1+\sum_{k\geq1}d_{k}^{-}e^{-ik\theta}.
\end{equation}
Now, recall that the complete homogeneous symmetric polynomials are a complete set of algebraically independent generators of the ring of symmetric functions. Thus, we can consider the specializations
\begin{equation*}
    h_{k}(x)\mapsto d_{k}^{+},\qquad h_{k}(y)\mapsto d_{k}^{-}\qquad (k\geq0)
\end{equation*}
on theorem \ref{mainth} to obtain an identity for the Toeplitz minors generated by an arbitrary function satisfying the conditions in Szeg\H o's theorem. Note also that the skew Schur polynomials above can be defined in terms of the Fourier coefficients $d_{k}^{+},d_{k}^{-}$ by means of the Jacobi-Trudi identities, so that the right hand side in \eqref{sumth} is well defined (note that the sum is actually finite for any fixed pair of partitions $\lambda$ and $\mu$). Therefore, we can rephrase theorem 1 as follows.
\begin{corollary}
Let $f(e^{i\theta})=\exp{(\sum_{k}c_{k}e^{ik\theta})}$, where the $c_{k}$ satisfy conditions \eqref{szcond}, and define $f^{+}$ and $f^{-}$ as in \eqref{factors}. Then,
\begin{equation} \label{sumth2}
    \lim_{N\to\infty}D_{N}^{\lambda,\mu}(f) = \exp{\left(\sum_{k=1}^{\infty}kc_{k}c_{-k}\right)}\sum_{\nu}s_{\lambda/\nu}(d^{-}_{k})s_{\mu/\nu}(d^{+}_{k}),
\end{equation}
where now the convergence is the usual convergence in $\mathbb{C}$, and we have denoted by $s_{\lambda/\mu}(d^{\pm}_{k})$ the determinants
\begin{equation*}
    s_{\lambda/\nu}(d^{\pm}_{k}) = \det{\left( d^{\pm}_{j-\nu_{j}-k+\lambda_{k}}\right)}_{j,k=1}^{\max{(l(\lambda),l(\nu))}}.
\end{equation*}
\end{corollary}
Similar examples where an algebraic result concerning Toeplitz determinants is seen to be equivalent to an analytic one for functions satisfying Szeg\H o's theorem can be found in \cite{TWwords,BorodinOkounkov}, for instance.

We have assumed in the above discussion that $f$ verifies the hypotheses in Szeg\H o's theorem. This was necessary in order for the limit $\lim_{N\rightarrow\infty}D_{N}(f)$ to be finite. Numerical experiments suggest however that \eqref{sumth2} holds for more general functions for which this limit is not finite, such as functions with Fisher-Hartwig singularities. It follows from a generalization of \eqref{BD} due to Lyons \cite{Lyons} that this is indeed true for the case of Toeplitz matrices generated by positive valued functions (as is the case, for instance, of Fisher-Hartwig singularities with zeros or poles, see section \ref{s.FH}), but we have been unable to extend this result to the most general case of arbitrary functions with Fisher-Hartwig singularities.

~

We conclude this section showing that exact formulas are available when the function $f$ can be obtained as a specialization with a finite number of nonzero variables. There are two possibilities:

\begin{itemize}
    \item Case 1: There is a factor of the type $H$ specialized to a finite set of variables. Suppose $f$ is of the form $f(z)=H(y_{1},\dots,y_{d};z^{-1})H(x;z)$. Then, in the same fashion as in Baxter's identity \eqref{baxterid}, the corresponding Toeplitz determinant stabilizes and we obtain the formula
    \begin{equation*}
        D_{N}^{\lambda,\mu}\left(\prod_{k=1}^{d}\frac{1}{1-y_{k}z^{-1}}\prod_{j=1}^{\infty}\frac{1}{1-x_{j}z}\right) = \prod_{k=1}^{d}\prod_{j=1}^{\infty}\frac{1}{1-x_{j}y_{k}}\sum_{\nu}s_{\lambda/\nu}(y)s_{\mu/\nu}(x),
    \end{equation*}
    that holds for every $N\geq d$. An analogous result holds for symbols of the type $f(z)=H(y_{1},\dots,y_{d};z^{-1})E(x;z)$.

    \item Case 2: There is a factor of the type $E$ specialized to a finite set of variables. We assume, without loss of generality, that $f$ is of the form $f(z)=E(y_{1},\dots,y_{d};z^{-1})E(x;z)$. As mentioned above, no $N$-independent formula is available for these symbols. However, it follows from \eqref{skewschur} that
    \begin{align} \begin{split} \label{EEschur}
        s_{((d^{N})+\mu/\lambda)'}(y_{1}^{-1},\dots,y_{d}^{-1},x) &= D_{N}^{\lambda,\mu}\left( \prod_{k=1}^{d} y_{k}^{-1} \prod_{k=1}^{d}(1+y_{k}z^{-1})\prod_{j=1}^{\infty}(1+x_{j}z) \right) = \\
        &= \prod_{k=1}^{d} y_{k}^{-N}D_{N}^{\lambda,\mu}(E(y_{1},\dots,y_{d};z^{-1})E(x;z)),
    \end{split} \end{align}
    and we see that in this case the Toeplitz minor can be expressed essentially as the specialization of a single skew Schur polynomial, indexed by the shape
    
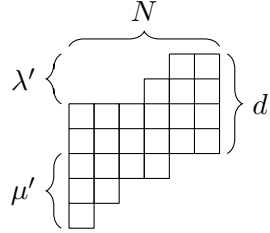
\begin{equation*}
        \begin{tikzpicture}[scale=0.33]
    
        \draw[step=1cm] (0,0) grid (1,5);
        \draw[step=1cm] (1,1) grid (2,5);
        \draw[step=1cm] (2,2) grid (3,5);
        \draw[step=1cm] (3,2) grid (4,6);
        \draw[step=1cm] (4,3) grid (6,7);
        
        \coordinate (mu1) at (0,0);
        \coordinate (mu2) at (0,3);
        \coordinate (d1) at (6,3); 
        \coordinate (N1) at (0,7);
        \coordinate (N2) at (6,7);
        \coordinate (l1) at (0,5); 
    
        \draw[decoration={brace,raise=.1cm,amplitude=.2cm},decorate] (mu1) -- node[left=.3cm]{$\mu'$} (mu2);
        \draw[decoration={brace,mirror,raise=.1cm,amplitude=.2cm},decorate] (d1) -- node[right=.3cm]{$d$} (N2);
        \draw[decoration={brace,raise=.1cm,amplitude=.2cm},decorate] (N1) -- node[above=.3cm]{$N$} (N2);
        \draw[decoration={brace,raise=.1cm,amplitude=.2cm},decorate] (l1) --        node[left=.3cm]{$\lambda'$} (N1);
        \end{tikzpicture}
    \end{equation*}
    A similar identity was obtained in \cite{Alexandersson}. Comparing with the analogue of equation \eqref{minorsum} for this symbol we see that \eqref{EEschur} coincides with
    \begin{equation*}
        \prod_{k=1}^{d}y_{k}^{-N}\sum_{\nu\subset{(N^{d})}}s_{\nu/\mu'}(y_{1},...,y_{d})s_{\nu/\lambda'}(x),
    \end{equation*}
    where the (finite) sum runs over all partitions $\nu$ satisfying $l(\nu)\leq N$ and $\nu_{1}\leq d$.
\end{itemize}


\section{Inverses of Toeplitz matrices and skew Schur polynomials} \label{s.inv}

The usual formula for the inversion of a matrix in terms of its cofactors reads as follows for the case of Toeplitz matrices
\begin{equation} \label{Tinv}
    \left(T_{N}^{-1}(f)\right)_{j,k} = (-1)^{j+k}\frac{D_{N-1}^{(1^{k-1}),(1^{j-1})}(f)}{D_{N}(f)}.
\end{equation}
Hence, whenever the inverse of a Toeplitz matrix is known explicitly, formula \eqref{Tinv} yields explicit evaluations of the formulas appearing in section \ref{s.res}. In particular, if the function $f$ is of the form $f(z)=E(y_{1},...,y_{d};z^{-1})E(x;z)$, the Toeplitz minor in the right hand side above has several expressions: in terms of the inverse of the corresponding Toeplitz matrix
\begin{align}
    D_{N}^{(1^{k}),(1^{j})}(f)&=(-1)^{j+k}D_{N+1}(f)(T_{N+1}^{-1}(f))_{j+1,k+1}, \label{mininv} \\
\intertext{as a specialization of a skew Schur polynomial}
    D_{N}^{(1^{k}),(1^{j})}(f)&= s_{(\underbrace{\scriptstyle N,...,N}_{d},j)/(k)}(y_{1}^{-1},\dots,y_{d}^{-1},x) \prod_{r=1}^{d}y_{r}^{N}, \label{ssinv} \\
\intertext{and as the multiple integral}
    D_{N}^{(1^{k}),(1^{j})}(f) &= \label{intinv} \\
    \frac{1}{N!}\frac{1}{(2\pi)^N}\int_{0}^{2\pi}...\int_{0}^{2\pi}e_{k}(e^{-i\theta_{1}}&,...,e^{-i\theta_{N}})e_{j}(e^{i\theta_{1}},...,e^{i\theta_{N}})\prod_{j=1}^{N}f(e^{i\theta_{j}})\prod_{1\leq j<k\leq N}|e^{i\theta_{j}}-e^{i\theta_{k}}|^{2}d\theta_{1}... d\theta_{N}, \nonumber
\intertext{where $e_{j},e_{k}$ are elementary symmetric polynomials \eqref{generat} (we assume in the three last identities that $N\geq 1$ and $0\leq j,k\leq N$). Moreover, formula \eqref{sumth} gives the asymptotic behaviour}
    \lim_{N\rightarrow\infty}D_{N}^{(1^{k}),(1^{j})}(f) &= \sum_{r=0}^{\min{(j,k)}}h_{k-r}(y)h_{j-r}(x) \lim_{N\rightarrow\infty} D_{N}(f), \label{mainsuminv}
\end{align}
where the convergence is as formal series or the usual convergence, according to the context (note that the partitions indexing the sum in \eqref{sumth} are now conjugated). It would be interesting to compare Widom’s asymptotic formula for the inverses of Toeplitz matrices \cite{Widom} with \eqref{mainsuminv}. Other work with formulas for the inverses of Toeplitz matrices is \cite{RambourS}.

In the following, we recall some known explicit inverses of Toeplitz matrices and compute another two in order to obtain evaluations for the Toeplitz minor $D_{N}^{(1^{k}),(1^{j})}(f)$. Comparing these with equations \eqref{ssinv} and \eqref{intinv} we will obtain explicit formulas for some specializations of the above skew Schur polynomial and for the above integral for some choices of the symbol $f$, as well as their asymptotics. We assume in the following invertibility of all the matrices involved.


\subsection{Tridiagonal Toeplitz matrices} \label{s.cheby}

A simple example is given by the Toeplitz matrix generated by the function $f(z)=E(y;z^{-1})E(x;z)$, where $x$ and $y$ are single (nonzero) variables
\begin{equation} \label{tridtoep}
	T_{N}(E(y;z^{-1})E(x;z))= \begin{pmatrix}
	1+xy & y & \\ x & 1+xy & \ddots \\	& \ddots & \ddots 
	\end{pmatrix}.
\end{equation} 
In \cite{Trench} an exact formula for the inverses of banded Toeplitz matrices was obtained (see \cite{Maximenko} for a recent work with a different proof). The inverse of a tridiagonal Toeplitz matrix has an expression in terms of Chebyshev polynomials of the second kind \cite{Szego}. These are defined by the recurrence relation
\begin{equation*}
	\begin{dcases}
	U_{n+1}(z)=2zU_{n}(z)-U_{n-1}(z) \qquad (n\geq 1),\\
	U_{0}(z)=1,\quad U_{1}(z)=2z.
	\end{dcases}
\end{equation*}
The determinant of the matrix \eqref{tridtoep} is then given by \cite{FonsecaPetronilho}
\begin{equation} \label{triddet}
    D_{N}(E(y;z^{-1})E(x;z)) = \frac{(xy)^{N+1}-1}{xy-1}= (xy)^{N/2}U_{N}(c)  \qquad \left(c=\frac{1+xy}{2\sqrt{xy}}\right),
\end{equation}
and its inverse by
\begin{equation*}
	(T_{N}^{-1}(E(y;z^{-1})E(x;z)))_{j,k} = \begin{dcases}
	(-1)^{j+k}\frac{y^{k-j}}{(xy)^{(k-j+1)/2}}\frac{U_{j-1}(c)U_{N-k}(c)}{U_{N}(c)} & (j\leq k), \\
	(-1)^{j+k}\frac{x^{j-k}}{(xy)^{(j-k+1)/2}}\frac{U_{k-1}(c)U_{N-j}(c)}{U_{N}(c)} & (j> k).
	\end{dcases}
\end{equation*}
Inserting these expressions in equation \eqref{ssinv} we obtain the following expression for an arbitrary skew Schur polynomial indexed by a shape of at most two rows and specialized to two variables
\begin{align*}
	s_{(N,j)/(k)}(x,y^{-1}) &= (xy^{-1})^{(N+j-k)/2}U_{\min{(j,k)}}(c)U_{N-\max{(j,k)}}(c) = \\
	&= \frac{1}{x^{k}y^{N+j-k}}\sum_{r=0}^{\min{(j,k)}}(xy)^{r}\sum_{r=\max{(j,k)}}^{N}(xy)^{r},
\end{align*}
for $j,k=0,...,N$ and $N\geq1$. It is well known that a Schur polynomial specialized to two variables is equal to a Chebyshev polynomial \cite{KadellSC}. We also obtain from formula \eqref{mainsuminv} that if $|x|,|y|<1$ then
\begin{equation*}
    \lim_{N\rightarrow\infty}s_{(N,j)/(k)}(x,y^{-1})y^{N} = x^{j}y^{k}\frac{(xy)^{-\min{(j,k)}-1}-1}{(xy)^{-1}-1}.
\end{equation*}


\subsection{The pure Fisher-Hartwig singularity} \label{s.FH}

The so-called pure Fisher-Hartwig singularity is the function
\begin{equation} \label{pureFH}
    |1-e^{i\theta}|^{2\alpha}e^{i\beta(\theta-\pi)}\qquad(0<\theta<2\pi),
\end{equation}
where the parameters $\alpha,\beta$ satisfy $\textrm{Re}(\alpha)>-1/2$ and $\beta\in\mathbb{C}$. The factor $|1-e^{i\theta}|^{2\alpha}$ may have a zero, a pole, or an oscillatory singularity at the point $z=1$, while the factor $e^{i\beta(\theta-\pi)}$ has a jump if $\beta$ is not an integer. Thus, depending on the different values of the parameters $\alpha$ and $\beta$, the symbol above may violate the regularity conditions in Szeg\H o's theorem. It will be more convenient to work with the equivalent factorization \cite{BottSilb}
\begin{equation*}
    (1-e^{i\theta})^{\gamma}(1-e^{-i\theta})^{\delta}.
\end{equation*}
This function coincides with \eqref{pureFH} if $\gamma=\alpha+\beta$ and $\delta=\alpha-\beta$; we will assume in the following that the parameters $\gamma$ and $\delta$ are positive integers. We can then express this function as the specialization
\begin{equation} \label{specFH}
    f(z)=\varphi_{\gamma,\delta}(z)=E(\underbrace{1,...,1}_{\delta};z^{-1})E(\underbrace{1,...,1}_{\gamma};z).
\end{equation}
Functions with general Fisher-Hartwig singularities are obtained as the product of a function verifying the regularity conditions in Szeg\H o's theorem times a finite number of translated pure singularities of the form $\varphi_{\gamma_{r},\delta_{r}}(e^{i(\theta-\theta_{r})})$. Each of these factors has a singularity with parameters $\gamma_{r},\delta_{r}$ at the point $e^{i\theta_{r}}$.

~

The inverse of the Toeplitz matrix generated by the pure FH singularity can be computed by means of the Duduchava-Roch formula \cite{Dudu,Roch,BottDR}
\begin{equation*}
    T((1-z)^{\gamma})M_{\gamma+\delta}T((1-z^{-1})^{\delta}) = \frac{\Gamma{(\gamma+1)}\Gamma{(\delta+1})}{\Gamma{(\gamma+\delta+1)}}M_{\delta}T(\varphi_{\gamma,\delta})M_{\gamma},
\end{equation*}
where $M_{a}$ is the diagonal matrix with entries $(M_{a})_{k,k}=\binom{a+k-1}{k-1}$, for $k\geq1$. B\"ottcher and Silbermann \cite{BottSilbFH} used this formula to give an explicit expression for the determinant of the Toeplitz matrix generated by the pure FH singularity
\begin{equation} \label{FHdet}
    D_{N}\left( \varphi_{\gamma,\delta} \right) = G(N+1)\frac{G(\gamma+\delta+N+1)}{G(\gamma+\delta+1)}\frac{G(\gamma+1)}{G(\gamma+N+1)} \frac{G(\delta+1)}{G(\delta+N+1)},
\end{equation}
where $G$ is the Barnes function \cite{Barnes}. Also the inverse of the corresponding Toeplitz matrix can be computed explicitly by means of this formula \cite{BottDR}
\begin{equation*}
    (T_{N}^{-1}(\varphi_{\gamma,\delta}))_{j,k}= (-1)^{j+k}\frac{\Gamma(\gamma+j)\Gamma(\delta+k)}{\Gamma(j)\Gamma(k)}\sum_{r=\max{(j,k)}}^{N}\frac{\Gamma(r)}{\Gamma(\gamma+\delta+r)}\binom{\gamma+r-k-1}{r-k}\binom{\delta+r-j-1}{r-j}.
\end{equation*}
Inserting these expressions in equation \eqref{ssinv} we obtain
\begin{align}
    &s_{(\underbrace{\scriptstyle N,...,N}_{d},j)/(k)}(1^{M}) =\, G(N+2)\frac{G(M+N+2)}{G(M+1)}\frac{G(M-d+1)}{G(M-d+N+2)}\frac{G(d+1)}{G(d+N+2)}\times \label{evskewFH} \\
    &\frac{\Gamma(M-d+j+1)}{\Gamma(j+1)}\frac{\Gamma(d+k+1)}{\Gamma(k+1)} \sum_{r=\max{(j,k)}}^{N}\frac{\Gamma(r+1)}{\Gamma(M+r+1)}\binom{M-d+r-k-1}{r-k}\binom{d+r-j-1}{r-j},  \nonumber
\end{align}
for $j,k\leq N$ and $M>d$ (or $M\geq d$, if $j=0$). 
The above formula recovers known evaluations whenever $k=0$ and thus the function in the left hand side above is a Schur polynomial (these can be computed by means of the hook-content formula \cite{Stanley}, for instance). Explicit expressions for such specialization of skew Schur polynomials indexed by partitions of certain shapes have been obtained recently in \cite{MPP}, and coincide with the above formula when the shapes are the same. The shapes covered by the above formula are not a subset nor a superset of those considered in \cite{MPP}.

~

Using expression \eqref{pureFH}, we see that the integral form of a Toeplitz minor generated by the pure Fisher-Hartwig generality
\begin{align}
    &D_{N}^{\lambda,\mu}(\varphi_{\gamma,\delta})=s_{((\delta^{N})+\mu/\lambda)'}(1^{\gamma+\delta}) = \label{morris} \\
    \frac{1}{N!}&\frac{1}{(2\pi)^{N}}\int_{0}^{2\pi}...\int_{0}^{2\pi}s_{\lambda}(e^{-i\theta})s_{\mu}(e^{i\theta})\prod_{j=1}^{N}e^{\frac{1}{2}i\theta_{j}(\gamma-\delta)}|1+e^{i\theta_{j}}|^{\gamma+\delta}\prod_{1\leq j<k\leq N}|e^{i\theta_{j}}-e^{i\theta_{k}}|^{2}d\theta_{1}...d\theta_{N}, \nonumber
\end{align}
is the unitary version of Selberg integral known as Morris integral, with the insertion of two Schur polynomials (we denote above $s_{\mu}(e^{i\theta})=s_{\mu}(e^{i\theta_{1}},\dots,e^{i\theta_{N}})$). Explicit formulas are known \cite{ForWar} for the evaluation of this integral, with and without the insertion of a single Schur polynomial $s_{\mu}$, although its expression as the specialization of a skew Schur polynomial \eqref{morris} appears to be new. We note also that its minor representation allows a direct computation for the case of a single polynomial
\begin{equation} \label{minorFH}
    D_{N}^{\varnothing,\mu}(\varphi_{\gamma,\delta}) = D_{N}(\varphi_{\gamma,\delta}) s_{\mu}(1^{N})\prod_{k=1}^{N}\frac{\Gamma(\gamma+k)}{\Gamma(\gamma+k-\mu_{k})}\frac{\Gamma(\delta+N-k+1)}{\Gamma(\delta+N-k+\mu_{k}+1)}.
\end{equation}
A proof of this identity is sketched in the appendix. An explicit expression of the integral \eqref{morris} with the insertion of two Schur polynomials is known for the case $\gamma=-\delta$ \cite{Kadell2}. Substituting $M-d$ by $\gamma$ and $d$ by $\delta$, formula \eqref{evskewFH} gives an explicit evaluation of this integral valid for general values\footnote{We have only proved the validity of the formula for integer values of $\gamma$ and $\delta$. However, by Carlson's theorem the formula holds for any positive $\gamma$ and $\delta$.} of $\gamma$ and $\delta$ whenever the Schur polynomials reduce to elementary symmetric polynomials $s_{\lambda}=e_{k}$, $s_{\mu}=e_{j}$.




\subsection{Principal specializations}

In order to study the principal specialization $x_{j}=q^{j-1}$ in the above formulas, we recall the well known method of Borodin for obtaining the inverse of the moment matrix of a biorthogonal ensemble. We follow the presentation in \cite{Borodin}, where details and proofs can be found. The starting point is a random matrix ensemble of the form
\begin{equation*}
    \int\dots\int \det{\left(\xi_{j}(z_{k})\right)_{j,k=1}^{N}}\det{\left(\eta_{j}(z_{k})\right)_{j,k=1}^{N}}\prod_{j=1}^{N}f(z_{j})dz_{j}
\end{equation*}
(up to a constant), for a weight function $f$ supported on some domain and two families of functions $(\xi_{j})$ and $(\eta_{j})$. If one is able to find two new families $(\zeta_{j})$ and $(\psi_{j})$ that biorthogonalize\footnote{Note that we are actually considering biorthonormal functions; we stick to the original terminology of \cite{Borodin} here and below and speak of biorthogonal functions in the following.} the former with respect to the weight $f$, that is
\begin{align} \begin{split} \label{biorth}
    \zeta_{j}\in\textrm{Span}&\{\xi_{1},\dots,\xi_{j}\}, \qquad \psi_{j}\in\textrm{Span}\{\eta_{1},\dots,\eta_{j}\}, \\
    &\int \zeta_{j}(z)\psi_{k}(z)f(z)dz=\delta_{j,k},
\end{split} \end{align}
then the matrix of coefficients of the kernel
\begin{equation} \label{kernel}
    K_{N}(z,\omega)=\sum_{r=1}^{N}\zeta_{r}(z)\psi_{r}(\omega) = \sum_{j,k=1}^{N}c_{j,k}\xi_{j}(z)\eta_{k}(\omega)
\end{equation}
satisfies
\begin{equation*}
    \left[(c_{j,k})_{j,k=1}^{N}\right]^{-1} = \left(\int \xi_{k}(z)\eta_{j}(z)f(z)dz\right)_{j,k=1}^{N}.
\end{equation*}
If the ensemble is an orthogonal polynomial ensemble, then the moment matrix on the right hand side above is a Hankel matrix, the functions $\xi_{j}$ and $\eta_{j}$ are the monomials $z^{j-1}$, and we have that $\zeta_{j}=\psi_{j}=p_{j}$, the orthogonal polynomials with respect to the weight function $f$, that is supported on the real line. The case where the moment matrix on the right hand side above is the Toeplitz matrix generated by a function $f$ supported on the unit circle corresponds to the biorthogonal ensemble with functions $\xi_{j}(z)=z^{-(j-1)}, \eta_{j}(z)=z^{j-1}$. Thus, the biorthogonality condition \eqref{biorth} amounts to finding two families of polynomials $p_{j}$ and $q_{j}$ such that
\begin{equation} \label{biorthpol}
    \frac{1}{2\pi}\int_{0}^{2\pi} p_{j}(e^{-i\theta})q_{k}(e^{i\theta})f(e^{i\theta})d\theta = \delta_{j,k}.
\end{equation}
Let us remark that only when the Toeplitz matrix is hermitian (that is, when the function $f$ is real valued), these polynomials verify $p_{j}(e^{-i\theta})=\overline{q_{j}(e^{i\theta})}$, the $q_{j}$ are the orthogonal polynomials with respect to $f$, and the kernel above is the usual Christoffel-Darboux kernel (see \cite{Baxter,Kailath} for more details). In general, one needs to consider a biorthogonal ensemble as above. Nevertheless, one can compute the polynomials $(p_{j})$ and $(q_{j})$ in a similar fashion to the orthogonal case.
\begin{lemma*}
    Suppose the determinants $D_{N}(f)$ are nonzero for every $N$. Then, the polynomials $p_{j}$ and $q_{j}$ in \eqref{biorthpol} are given by
    \begin{align*}
        p_{j}(z)&= \frac{1}{(D_{j}(f)D_{j+1}(f))^{1/2}} \begin{vmatrix}
        d_{0} & d_{-1} & \dots & d_{-j} \\
        d_{1} & d_{0} & \dots & d_{-(j-1)} \\
        \vdots & \vdots & & \vdots & \\
        d_{j-1} & d_{j-2} & & d_{-1} \\
        1 & z & \dots & z^{j} \end{vmatrix}, \\ q_{j}(z)&= \frac{1}{(D_{j}(f)D_{j+1}(f))^{1/2}} \begin{vmatrix}
        d_{0} & d_{-1} & \dots & d_{-(j-1)} & 1 \\
        d_{1} & d_{0} & \dots & d_{-(j-2)} & z \\
        \vdots & \vdots & & \vdots & \vdots \\
        d_{j} & d_{j-1} & \dots & d_{1} & z^{j} \end{vmatrix}.
    \end{align*}
\end{lemma*}
\begin{proof}
The condition on the determinants implies the existence of the polynomials themselves (see Proposition 2.9 in \cite{Borodin}, for instance), and they are uniquely determined up to multiplicative constants. Hence, it suffices to verify the biorthogonality condition \eqref{biorthpol}. We denote
\begin{equation} \label{pols}
    p_{j}(z)=\sum_{r=0}^{j}a^{(j)}_{r}z^{r}, \qquad q_{k}(z)=\sum_{r=0}^{k}b^{(k)}_{r}z^{r}.
\end{equation}
Now, if $j\geq k$ in \eqref{biorthpol} we can rewrite this integral as the sum
\begin{align*}
    \frac{1}{2\pi}\int_{0}^{2\pi}& p_{j}(e^{-i\theta})q_{k}(e^{i\theta})f(e^{i\theta})d\theta = \frac{1}{(D_{k}(f)D_{k+1}(f)D_{j}(f)D_{j+1}(f))^{1/2}}\times \\
    &\sum_{r=0}^{k}b^{(k)}_{r}\begin{vmatrix}
        d_{0} & d_{-1} & \dots & d_{-j} \\
        d_{1} & d_{0} & \dots & d_{-(j-1)} \\
        \vdots & \vdots & & \vdots & \\
        d_{j-1} & d_{j-2} & & d_{-1} \\
        \frac{1}{2\pi}\int_{0}^{2\pi}e^{ir\theta}f(e^{i\theta})d\theta & \frac{1}{2\pi}\int_{0}^{2\pi}e^{i(r-1)\theta}f(e^{i\theta})d\theta & \dots & \frac{1}{2\pi}\int_{0}^{2\pi}e^{i(r-j)\theta}f(e^{i\theta})d\theta \end{vmatrix},
\end{align*}
which vanishes if $j>k$ and equals $1$ if $j=k$, since the last row in the above determinants is precisely $(d_{r},d_{r-1},\dots,d_{r-j})$. Analogously, if $j<k$ in \eqref{biorthpol} the integral equals
\begin{equation*}
    \frac{1}{(D_{k}(f)D_{k+1}(f)D_{j}(f)D_{j+1}(f))^{1/2}} \sum_{r=0}^{j}a^{(j)}_{r}\begin{vmatrix}
        d_{0} & d_{-1} & \dots & d_{-(k-1)} & \frac{1}{2\pi}\int_{0}^{2\pi}e^{-ir\theta}f(e^{i\theta})d\theta \\
        d_{1} & d_{0} & \dots & d_{-(k-2)} & \frac{1}{2\pi}\int_{0}^{2\pi}e^{-i(r-1)\theta}f(e^{i\theta})d\theta \\
        \vdots & \vdots & & \vdots & \vdots \\
        d_{k} & d_{k-1} & \dots & d_{1} & \frac{1}{2\pi}\int_{0}^{2\pi}e^{-i(r-k)\theta}f(e^{i\theta})d\theta \end{vmatrix},
\end{equation*}
and again all the determinants in the sum vanish.
\end{proof}

We now use this result to study the principal specialization of skew Schur polynomials indexed by the shapes considered earlier. We assume in the following that $q$ is a new (real) variable verifying $|q|<1$. We will denote by $\Gamma_{q}$ and $G_{q}$ the $q$-Gamma and $q$-Barnes functions \cite{Nishizawa}, that in particular verify
\begin{equation} \label{qfuncts}
    \Gamma_{q}(k+1)=\frac{\prod_{j=1}^{k}(1-q^{j})}{(1-q)^{k}}=\frac{(q;q)_{k}}{(1-q)^{k}}, \qquad G_{q}(k+1)=\prod_{j=1}^{k-1}\Gamma_{q}(j+1),
\end{equation}
whenever $k$ is a natural number (we assume that an empty product takes the value $1$). The $q$-binomial coefficient is then given by
\begin{align*}
     \qbinom{\omega}{z}=\frac{\Gamma_{q}(\omega+1)}{\Gamma_{q}(z+1)\Gamma_{q}(\omega-z+1)} \qquad\left(\textrm{Re}(\omega)\geq\textrm{Re}(z)>0\right).
\end{align*}
These functions coincide with their classical counterparts in the $q\rightarrow1$ limit, that is
\begin{equation*}
    \lim_{q\rightarrow1}\Gamma_{q}(z)=\Gamma(z), \qquad \lim_{q\rightarrow1}G_{q}(z)=G(z), \qquad \lim_{q\rightarrow1}\qbinom{\omega}{z}=\binom{\omega}{z},
\end{equation*}
for all the $\omega$ and $z$ such that the right hand sides above make sense. We consider the following specialization \cite{FoataHan}
\begin{equation*}
    f(z)=\Theta_{\gamma,\delta}(z)=E(1,q,\dots,q^{\delta-1};z^{-1})E(q,q^{2},\dots,q^{\gamma};z) = \sum_{k=-\delta}^{\gamma}\qbinom{\delta+\gamma}{\delta+k}q^{k(k+1)/2}z^{k},
\end{equation*}
for some positive integers $\gamma$ and $\delta$. The Toeplitz determinant generated by this function equals
\begin{equation*}
    D_{N}(\Theta_{\gamma,\delta}) = G_{q}(N+1)\frac{G_{q}(\delta+\gamma+N+1)}{G_{q}(\delta+\gamma+1)}\frac{G_{q}(\delta+1)}{G_{q}(\delta+N+1)} \frac{G_{q}(\gamma+1)}{G_{q}(\gamma+N+1)},
\end{equation*}
and the biorthogonal polynomials $p_{j},q_{j}$ are given by
\begin{align} \begin{split} \label{pspecpols}
    p_{j}(z)&=\left(\frac{(q;q)_{\delta+j}(q;q)_{\gamma+j}}{(q;q)_{j}(q;q)_{\delta+\gamma+j}}\right)^{1/2}\sum_{r=0}^{j}(-1)^{j+r}\qbinom{j}{r}\frac{(q;q)_{\gamma+r}}{(q;q)_{\gamma+j}}\frac{(q;q)_{\delta+j-r-1}}{(q;q)_{\delta-1}}z^{r}, \\
    q_{j}(z)&=\left(\frac{(q;q)_{\delta+j}(q;q)_{\gamma+j}}{(q;q)_{j}(q;q)_{\delta+\gamma+j}}\right)^{1/2}\sum_{r=0}^{j}(-1)^{j+r}\qbinom{j}{r}\frac{(q;q)_{\gamma+j-r-1}}{(q;q)_{\gamma-1}}\frac{(q;q)_{\delta+r}}{(q;q)_{\delta+j}}q^{r}z^{r},
\end{split} \end{align}
where $(q;q)_{k}$ is as defined in \eqref{qfuncts}. The last three identities can be proved directly from their determinantal expressions. We do not include the computations here but point to the second method of proof in \cite{BWFH}, that can be generalized to the present setting. A similar computation is included in the appendix as an example. Recalling the notation \eqref{pols}, we have that the kernel \eqref{kernel} is then given by
\begin{align*}
    &K_{N+1}(z,\omega)=\sum_{r=0}^{N}p_{r}(z)q_{r}(\omega^{-1})= \sum_{j,k=0}^{N}\left(\sum_{r=\max{(j,k)}}^{N}a_{j}^{(r)}b_{k}^{(r)}\right)z^{j}\omega^{-k} = \\
    &\sum_{j,k=0}^{N}\left(\sum_{r=\max{j,k}}^{N}(-1)^{j+k}q^{j}\frac{\Gamma_{q}(\delta+j+1)\Gamma_{q}(\gamma+k+1)\Gamma_{q}(r+1)}{\Gamma_{q}(j+1)\Gamma_{q}(k+1)\Gamma_{q}(\delta+\gamma+r+1)}\qbinom{\gamma+r-k-1}{r-k}\qbinom{\delta+r-j-1}{r-j} \right)z^{j}\omega^{-k}.
\end{align*}
Moreover, the coefficient of $z^{j}\omega^{-k}$ in the above sum is the $(j+1,k+1)$-th entry of the inverse of the matrix $T_{N+1}(\Theta_{\gamma,\delta})$. Inserting this into expression \eqref{ssinv} we obtain
\begin{align*}
    &s_{(\underbrace{\scriptstyle N,...,N}_{d},j)/(k)}(1,q,\dots,q^{M-1}) =  \\ &q^{dj-(d-1)k+d(d-1)N/2}G_{q}(N+2)\frac{G_{q}(M+N+2)}{G_{q}(M+1)}\frac{G_{q}(M-d+1)}{G_{q}(M-d+N+2)}\frac{G_{q}(d+1)}{G_{q}(d+N+2)}\times  \\
    & \sum_{r=\max{(j,k)}}^{N}\frac{\Gamma_{q}(M-d+j+1)\Gamma_{q}(d+k+1)\Gamma_{q}(r+1)}{\Gamma_{q}(j+1)\Gamma_{q}(k+1)\Gamma_{q}(M+r+1)}\qbinom{M-d+r-k-1}{r-k}\qbinom{d+r-j-1}{r-j},  
\intertext{for $j,k\leq N$ and $M>d$ (or $M\geq d$, if $j=0$). As expected, this expression coincides with \eqref{evskewFH} in the $q\rightarrow1$ limit. Also, as above, the formula recovers known expressions whenever $k=0$ (and thus we have a Schur polynomial, comparing again with the hook-content formula \cite{Stanley}, for instance). Finally, it follows from \eqref{mainsuminv} and the Cauchy identity that}
    &\lim_{N\rightarrow\infty}s_{(\underbrace{\scriptstyle N,...,N}_{d},j)/(k)}(1,q,\dots,q^{M-1})q^{-Nd(d-1)/2} = \\
    &\frac{q^{dj-(d-1)k}}{(1-q)^{d(M-d)}}\frac{G_{q}(d+1)G_{q}(M-d+1)}{G_{q}(M+1)}\sum_{r=0}^{\min{(j,k)}}q^{-r}\qbinom{M-d+j-r-1}{j-r}\qbinom{d+k-r-1}{k-r}.
\end{align*}
Note that inverting a Toeplitz matrix by means of the kernel \eqref{kernel} is a general procedure that can be used to obtain explicit evaluations of other specializations of the skew Schur polynomials of the shapes considered above, as long as the biorthogonal polynomials \eqref{biorthpol} are available. In particular, the results in subsection \ref{s.FH} for the pure Fisher-Hartwig singularity can be obtained in such a way. The biorthogonal polynomials can be obtained\footnote{In the hermitian case $\gamma=\delta$, where the polynomials are a single family of orthogonal polynomials, one recovers the family $S_{n}^{\,a}(z)$ introduced in \cite{Askey} after substituting $q$ by $q^{1/2}$, $z$ by $q^{-1/2}z$ and $a$ by $q^{\gamma}$.} as the $q\rightarrow1$ limit of the polynomials \eqref{pspecpols}, leading to the same formula \eqref{evskewFH}. 

~

Finally, taking into account that only one set of variables in the specialization of $f$ needs to be finite in equations \eqref{mininv}-\eqref{mainsuminv}, we can study the principal specialization of the above skew Schur polynomials with an infinite number of variables. To do so, we consider the specialization
\begin{equation*}
    f(z)=\Theta_{\delta}(z)=E(1,q^{-1},\dots,q^{-(\delta-1)};z^{-1})E(q^{\delta},q^{\delta+1},\dots;z)=\sum_{k=-\delta}^{\infty}\frac{q^{k\delta+k(k-1)/2}}{(q;q)_{\delta+k}}z^{k},
\end{equation*}
for some positive integer $\delta$. The Toeplitz determinant generated by this function is
\begin{equation*}
    D_{N}(\Theta_{\delta})=\frac{1}{(1-q)^{\delta N}}\frac{G_{q}(\delta+1)G_{q}(N+1)}{G_{q}(\delta+N+1)},
\end{equation*}
and the biorthogonal polynomials on the unit circle with respect to this function are given by
\begin{align*}
    p_{j}(z)&=\left(\frac{(q;q)_{\delta+j}}{(q;q)_{j}}\right)^{1/2}\sum_{r=0}^{j}(-1)^{j+r}\qbinom{j}{r}\frac{(q;q)_{\delta+j-r-1}}{(q;q)_{\delta-1}}q^{-(\delta-1)(j-r)}z^{r}, \\
    q_{j}(z)&=\left(\frac{1}{(q;q)_{j}(q;q)_{\delta+j}}\right)^{1/2}\sum_{r=0}^{j}(-1)^{j+r}\qbinom{j}{r}(q;q)_{\delta+r}q^{\delta(j-r)}z^{r}.
\end{align*}
Again, these expressions can be verified from their determinantal formulas. The kernel in this case is then
\begin{equation*}
    K_{N+1}(z,\omega)=\sum_{j,k=0}^{N}\left(\sum_{r=\max{j,k}}^{N}(-1)^{j+k}q^{r+(\delta-1) j-\delta k}\frac{(q;q)_{\delta+k}}{(q;q)_{j}}\qbinom{r}{r-k}\qbinom{\delta+r-j-1}{r-j} \right)z^{j}\omega^{-k}.
\end{equation*}
Inserting this in equation \eqref{ssinv} we arrive at
\begin{align*}
    &s_{(\underbrace{\scriptstyle N,...,N}_{d},j)/(k)}(1,q,\dots) = \\
    &\frac{q^{(d-1)j-dk+d(d-1)N/2}}{(1-q)^{d(N+1)}}\frac{G_{q}(N+2)G_{q}(d+1)}{G_{q}(d+N+2)}\frac{(q;q)_{d+k}}{(q;q)_{j}}\sum_{r=\max{(j,k)}}^{N}q^{r}\qbinom{r}{r-k}\qbinom{d+r-j-1}{r-j}.
\intertext{Once again, this identity coincides with the one given by the hook-content formula for $k=0$. It follows from \eqref{mainsuminv} and the Cauchy identity that}
    &\lim_{N\rightarrow\infty}s_{(\underbrace{\scriptstyle N,...,N}_{d},j)/(k)}(1,q,\dots)q^{-Nd(d-1)/2} = \\
    &q^{dj-(d-1)k}\frac{(1-q)^{d(d-1)/2}G_{q}(d+1)}{(q;q)_{\infty}^{d}}\sum_{r=0}^{\min{(j,k)}}q^{-r}\frac{1}{(q;q)_{j-r}}\qbinom{d+k-r-1}{k-r},
\end{align*}
where $(q;q)_{\infty}=\prod_{k=1}^{\infty}(1-q^{k})$ denotes the Euler function.


\begin{acknowledgements}
We thank Jorge Lobera for a MatLab implementation of formula \eqref{BD} and Alexandra Symeonides and T\^ania Zaragoza for useful discussions. We also thank an anonymous referee for several helpful remarks. The work of DGG was supported by the Fundação para a Ciência e a Tecnologia through the LisMath scholarship PD/BD/113627/2015. The work of MT was partially supported by the Fundação para a Ciência e a Tecnologia through its program Investigador FCT IF2014, under contract IF/01767/2014. The  work is also partially supported by FCT Project PTDC/MAT-PUR/30234/2017.
\end{acknowledgements}


\section*{Appendix: Direct computation of a minor generated by the pure FH singularity}

We now sketch a proof of identity \eqref{minorFH}. We follow the second of the two proofs given in \cite{BWFH} for the corresponding Toeplitz determinant. We include this computation to showcase how the Toeplitz minor structure can be exploited to obtain evaluations of the more complicated objects considered (i.e. multiple integrals, skew Schur polynomials), rather than for its mathematical insight.

The Fourier coefficients of $\varphi_{\gamma,\delta}$ are \cite{BottSilb}
\begin{equation*}
    d_{k}=\frac{\Gamma(\gamma+\delta+1)}{\Gamma(\gamma-k+1)\Gamma(\delta+k+1)}.
\end{equation*}
After taking out the factors
\begin{equation*}
   \prod_{j=1}^{N}\frac{\Gamma(\gamma+\delta+1)}{\Gamma(\gamma-\mu_{N}+N-j+1)},\qquad\prod_{k=1}^{N}\frac{1}{\Gamma(\delta+\mu_{k}+N-k+1)},
\end{equation*}
coming from the rows and columns of $D_{N}^{\varnothing,\mu}(\varphi_{\gamma,\delta})$ respectively, we obtain the determinant
\begin{equation} \label{detcomp} \begin{vmatrix} 
    \frac{\Gamma(\gamma-\mu_{N}+N)}{\Gamma(\gamma-\mu_{1}+1)} \frac{\Gamma(\delta+\mu_{1}+N)}{\Gamma(\delta+\mu_{1}+1)} && \frac{\Gamma(\gamma-\mu_{N}+N)}{\Gamma(\gamma-\mu_{2}+2)} \frac{\Gamma(\delta+\mu_{2}+N-1)}{\Gamma(\delta+\mu_{2})} && \dots &&  \frac{\Gamma(\delta+\mu_{N}+1)}{\Gamma(\delta+\mu_{N}-N+2)} \\
    \frac{\Gamma(\gamma-\mu_{N}+N-1)}{\Gamma(\gamma-\mu_{1})} \frac{\Gamma(\delta+\mu_{1}+N)}{\Gamma(\delta+\mu_{1}+2)} && \frac{\Gamma(\gamma-\mu_{N}+N-1)}{\Gamma(\gamma-\mu_{2}+1)} \frac{\Gamma(\delta+\mu_{2}+N-1)}{\Gamma(\delta+\mu_{2}+1)} && \dots && \frac{\Gamma(\delta+\mu_{N}+1)}{\Gamma(\delta+\mu_{N}-N+3)} \\
    \vdots && \vdots && && \vdots \\
   \frac{\Gamma(\gamma-\mu_{N}+1)}{\Gamma(\gamma-\mu_{1}-N+2)} && \frac{\Gamma(\gamma-\mu_{N}+1)}{\Gamma(\gamma-\mu_{1}-N+3)} && \dots && 1
\end{vmatrix}. \end{equation}
Subtracting $(\delta+\mu_{N}-N+1+j)$ times the $(j+1)$-th row from the $j$-th row, for $j=1,...,N-1,$ we can make the last column vanish except for the $1$ at the bottom, thus obtaining a determinant of order $N-1$. After extracting the factor
\begin{equation*}
    \prod_{k=1}^{N-1}(\gamma+\delta+1)(\mu_{k}-\mu_{N}+N-k)
\end{equation*}
from the columns of the matrix, and the factor
\begin{equation*}
    \prod_{j=1}^{N-1}\frac{\Gamma(\gamma-\mu_{N}+j)}{\Gamma(\gamma-\mu_{N-1}+j)}
\end{equation*}
from the rows, we obtain a determinant with the same structure as \eqref{detcomp}, but with the following changes: $N$ is replaced by $N-1$, $\delta$ is replaced by $\delta+1$ and $\mu$ is replaced by the partition $(\mu_{1},\dots,\mu_{N-1})$, that results from discarding the last part of $\mu$. Making use of this recursive structure and the well-known expression
\begin{equation} \label{schurat1}
    s_{\mu}(1^{N})=\frac{1}{G(N+1)}\prod_{1\leq j<k\leq N}(\mu_{j}-\mu_{k}+k-j) \qquad (N\geq l(\mu)),
\end{equation}
one arrives at the desired expression.


\end{document}